\documentclass[leqno,12pt]{amsart} 
\usepackage{amssymb}
\usepackage{amsmath}
\usepackage{amsthm}
\usepackage{amscd}
\usepackage{graphicx}
\usepackage[dvips]{color}
\usepackage[all]{xy}
\theoremstyle{plain} 
\newtheorem{theorem}{\indent\sc Theorem}[section]
\newtheorem{lemma}[theorem]{\indent\sc Lemma}
\newtheorem{corollary}[theorem]{\indent\sc Corollary}
\newtheorem{proposition}[theorem]{\indent\sc Proposition}

\theoremstyle{definition} 
\newtheorem{definition}[theorem]{\indent\sc Definition}
\newtheorem{remark}[theorem]{\indent\sc Remark}
\newtheorem{example}[theorem]{\indent\sc Example}

\newcommand{\ddbar}{\partial \bar{\partial}}
\newcommand{\dbar}{\bar{\partial}}

\newcommand{\tensor}{\otimes \det}
\newcommand{\wtilde}{\widetilde}
\newcommand{\what}{\widehat}
\begin{document}

\title[$L^2$ estimates and vanishing theorems]
{$L^2$ estimates and vanishing theorems for holomorphic vector bundles equipped with singular Hermitian metrics} 

\author[T. Inayama]{Takahiro Inayama} 

\subjclass[2010]{ 
Primary 32J25; Secondary 32L20.
}
%
\keywords{ 
Singular Hermitian metrics on vector bundles, $L^2$ estimates, cohomology vanishing.
}
\address{
Graduate School of Mathematical Sciences, The University of Tokyo \endgraf
3-8-1 Komaba, Meguro-ku, Tokyo, 153-8914 \endgraf
Japan
}
\email{inayama@ms.u-tokyo.ac.jp}

\maketitle

\begin{abstract}
We investigate singular Hermitian metrics on vector bundles, especially strictly Griffiths positive ones.
$L^2$ estimates and vanishing theorems usually require an assumption that vector bundles are Nakano positive. However,
there is no general definition of the Nakano positivity in singular settings.
In this paper, we show various $L^2$ estimates and vanishing theorems by assuming that
the vector bundle is strictly Griffiths positive and the base manifold is projective.
\end{abstract}

\section{Introduction}

We investigate singular Hermitian metrics on vector bundles.
Singular Hermitian metrics on line bundles have a key role in complex geometry.
They make it possible that we apply complex analytic methods to complex algebraic geometry (cf. \cite{DemAna}).
Singular Hermitian metrics on vector bundles were also introduced and investigated
in many papers (for examples, \cite{BP}, \cite{deC}, \cite{HPS}, \cite{PT}, \cite{Rau1}, etc.).
However, it is known that curvature currents of
singular Hermitian metrics on vector bundles are not always defined with measure coefficients \cite[Theorem 1.5]{Rau1}.
As a result, a positivity of singular Hermitian metrics on vector bundles generally cannot
be dealt with directly by using the curvature currents. Griffiths semi-positivity or
semi-negativity of singular Hermitian metrics is defined without using the curvature currents (\cite{BP}, \cite{PT}, \cite{Rau1}, see Definition \ref{GrifPos}).
Nevertheless, a general definition of Nakano positivity has not been formulated
even though $L^2$ estimates and vanishing theorems usually require an assumption that vector bundles are Nakano positive.

In this paper, we show various $L^2$ estimates and vanishing theorems.
To be precise, we have the following result. We let
$X$ be an $n$-dimensional complex projective manifold, let $\omega $ be a K\"ahler form on $X$, let $dV_\omega=\frac{\omega^n}{n!}$ be the volume form determined by $\omega$,
and let $E\to X$ be a holomorphic vector bundle of rank $r$ over $X$.
For simplicity, unless otherwise stated, we fix the K\"ahler metric $\omega$ on the projective manifold $X$ as $\omega=\sqrt{-1}\Theta_{h_A}$ for some ample line bundle $(A, h_A)$ (since $\omega$ is only referenced to define strict Griffiths $\delta_\omega$-positivity on $X$, this does not alter the situation in any meaningful way).

\begin{theorem}\label{mainthm1}
Let $h$ be a Griffiths semi-positive singular Hermitian metric on $E$. We assume that there exists a proper analytic subset $S$ such that 
$h$ is strictly Griffiths $\delta_\omega$-positive on $X\setminus S$ in the sense of Definition \ref{def:generalnegativity}. 
Suppose that $f$ is an $E$-valued $(n,n)$-form
with finite $L^2$-norm. Then there is an $E$-valued $(n,n-1)$-form $g$ such that
$$
\dbar g=f, \hspace{5mm}\int_X | g |^2_{h,\omega} dV_\omega \le \frac{1}{\delta n}\int_X | f|^2_{h,\omega} dV_\omega .
$$
\end{theorem}

Roughly speaking, the singular Hermitian metric $h$ is strictly Griffiths $\delta$-positive (or simply strictly Griffiths positive)
if the curvature current $\Theta_{h^\star}$ of the dual Hermitian vector bundle $(E^\star, h^\star)$ satisfies the following inequality
$$
\sum_{j, k=1}^n (\Theta_{j\overline{k}}^{h^\star} s, s)_h\xi_j \overline{\xi}_k \le -\delta \| s \|^2_{h^\star} |\xi|^2
$$
in the sense of
distributions for any section $s$ of $E^\star$ and any vector field $\xi = \sum \xi_j \frac{\partial}{\partial z_j}$,
where $\sqrt{-1}\Theta_h^\star = \sum_{j, k=1}^n \Theta_{j\overline{k}}^{h^\star} dz_j\wedge d\overline{z}_k$ (see Definition \ref{sGp}).
This type of theorem was announced by H.~Raufi \cite{Rau2} in the case $n=1$. In \cite{Rau2}, $L^2$ estimates
are established on Riemann surfaces for the reason that Griffiths positivity of vector bundles coincides with Nakano positivity
of those on Riemann surfaces. As we have said above, we allow the dimension of the base manifold $X$ to be larger than $1$.
As an application of Theorem \ref{mainthm1}, we have the next statement.

\begin{corollary}\label{maincor1}
Let $h$ be a Griffiths semi-positive singular Hermitian metric such that $h$ is strictly Griffiths $\delta_\omega$-positive on $X\setminus S$, 
and $K_X$ be the canonical bundle of $X$.
If the Lelong number $\nu(-\log \det h, x)<2$  for all points $x\in X$,
the $n$-th cohomology group of $X$ with coefficients in the sheaf of holomorphic sections of $K_X\otimes E$ vanishes $\colon$
$$
H^n(X, K_X\otimes E)=0 .
$$
\end{corollary}

Here $-\log \det h$ is locally plurisubharmonic.
Then the Lelong number $\nu(-\log \det h, x)$ is naturally defined (see Definition \ref{linelelong}).

We show another $L^2$ estimate. Let $h$ be a smooth
Hermitian metric on $E$. If $(E,h)$ is positive in the sense of Griffiths,
$(E\otimes \det E, h\otimes \det h)$ is positive in the sense of Nakano.
This is a well-known result of \cite{DSk}. As described above, Nakano positivity of
a singular Hermitian metric has not been defined. However, we show some $L^2$ estimate
by applying this result to a vector bundle equipped with a singular Hermitian metric.

\begin{theorem}\label{mainthm2}
Let $h$ be a Griffiths semi-positive singular Hermitian metric  on $E$ such that $h$ is strictly Griffiths $\delta_\omega$-positive on $X\setminus S$. 
Suppose that $q$ is a positive integer, $f$ is an $E\otimes \det E$-valued $\dbar $-closed $(n,q)$-form with finite $L^2$ norm.
Then there is an $E\otimes {\rm det}E$-valued $(n,q-1)$-form $g$ such that

$$
\dbar g=f, \hspace{5mm}\int_X | g |^2_{h\tensor h,\omega} dV_\omega \le \frac{1}{\delta qr}\int_X | f|^2_{h\tensor h,\omega} dV_\omega .
$$
\end{theorem}

As is the case in Corollary \ref{maincor1}, we have the following statement.

\begin{corollary}\label{maincor2}
Let $h$ be a Griffiths semi-positive singular Hermitian metric  on $E$ such that $h$ is strictly Griffiths $\delta_\omega$-positive on $X\setminus S$. 
If the Lelong number $\nu(- \log \det h, x)<1$  for all points $x\in X$, then
$$
H^q(X, K_X\otimes E\otimes \det E)=0 .
$$
\end{corollary}

Corollary \ref{maincor2} is a generalization of the so-called Griffiths vanishing theorem in the smooth setting (cf. \cite[Chapter VII, Corollary 9.4]{DemCom}).
For a (strictly) Griffiths positive singular Hermitian metric $h$, we cannot conclude that $h\otimes \det h$ is (strictly) Nakano positive for the reason that the definition of Nakano positive singular Hermitian metrics has not been formulated.
However, $h\otimes \det h$ behaves like a Nakano positive metric as in Theorem \ref{mainthm2} and Corollary \ref{maincor2}.
Corollary \ref{maincor2} can be used to discriminate the existence of a Griffiths positive singular
Hermitian metric on a certain vector bundle. We have the following example.

\begin{example}\label{exampleshm}
Let $V$ be a complex vector space of dimension $n+1$, and $X=P(V):=(V\setminus \{ 0 \})/\mathbb{C}^\star$
be the projective space of $V$.
We also let $\mathcal{O}(-1)$ denote the
tautological line bundle, and $Q$ denote the
quotient bundle $V/\mathcal{O}(-1)$. Then there do not exist any Griffiths semi-positive singular Hermitian metrics 
satisfying the strict Griffiths $\delta_\omega$-positivity outside some proper analytic subset $S$ and $\nu(- \log \det h, x)<1$ for all $x\in X$ on $Q$.

\end{example}

It is known that the above vector bundle $Q$ has a Griffiths semipositive smooth Hermitian metric, but
does not have any strictly Griffiths positive ones (cf. \cite[Chapter VII, Example 8.4]{DemCom}). Roughly speaking, Example \ref{exampleshm} asserts that there
cannot exist any strictly Griffiths positive Hermitian metrics even if they are (mildly) singular.
The criterion which determines whether a Griffiths positive singular Hermitian metric of a vector bundle exist
is one of important applications of Corollary \ref{maincor2}.

The organization of this paper is as follows. In Section \ref{prelim}, we explain
basic definitions and properties of singular Hermitian metrics on vector bundles.
In Section \ref{sHmsection}, we prepare some lemmas and properties about Griffiths positive or negative
singular Hermitian metrics on vector bundles.
In Section \ref{proofthm1} and \ref{proofthm2}, we prove the main theorems and corollaries.

\vskip3mm
{\bf Acknowledgment. }
The author would like to thank his supervisor Prof. Shigeharu Takayama
for inspiring and helpful comments. 
He is also grateful to anonymous referees for their suggestions to improve the manuscript.

\section{Singular Hermitian metrics on vector bundles}\label{prelim}
In this section, we introduce a definition and property of a singular Hermitian metric on a vector bundle.
To start with, we refer to basic notions of smooth Hermitian metrics on vector bundles. Throughout this section,
$X$ denotes a complex manifold with a positive Hermitian form $\omega$, and $E$ denotes a holomorphic vector bundle over $X$.

\subsection{Positivity concepts of smooth Hermitian metrics on vector bundles}

Let $h$ be a smooth Hermitian metric on $E$ and the Chern curvature of $(E, h)$ be $\Theta_{E, h}$ or simply $\Theta$.
The curvature $\Theta_{E, h}$ is a ${\rm Hom}(E, E)$ valued $(1, 1)$ form, thus let $(z_1,\ldots z_n)$ be holomorphic coordinates on $X$ around $p$ and
let $(e_1,\ldots e_r)$ be an orthonormal frame of $E$.
By referring to \cite[(3.8)]{DemAna}, writing
$$
\sqrt{-1}\Theta_{E, h}= \sum_{1\le j, k\le n, 1\le \lambda, \mu \le r} c_{j \overline{k} \lambda \mu}dz_j\wedge d\overline{z}_k \otimes e_\lambda^\star \otimes e_\mu ,
$$
we can identify the curvature tensor to a Hermitian form
$$
\Theta_{E, h}(\xi \otimes v, \xi \otimes v)=\sum_{1\le j, k\le n, 1\le \lambda, \mu \le r} c_{j \overline{k} \lambda \mu}\xi_j \overline{\xi}_k v_\lambda \overline{v}_\mu
\hspace{5mm}(\xi = \sum_{j=1}^n \xi_{j} \frac{\partial}{\partial z_j} \in T_X, v=\sum_{\lambda =1}^r v_\lambda e_\lambda \in E)
$$
on $T_X\otimes E$. This leads to positivity concepts.

\begin{definition}
The smooth Hermitian vector bundle $(E, h)$ is \\
(1) {\it Griffiths semi-positive} (resp. {\it Griffiths positive}) if $\Theta_{E, h}(\xi \otimes v, \xi \otimes v) \ge 0$ (resp. $>0$)
for every local non-zero section $\xi \in T_X , v\in E$. We denote it by $\Theta_{E, h}\ge_{\rm Grif} 0$ (resp. $\Theta_{E, h}>_{\rm Grif} 0$). \\
(2) {\it Nakano semi-positive} (resp. {\it Nakano positive}) if $\Theta_{E, h}(\tau, \tau) \ge 0$ (resp. $>0$)
for all non-zero tensors $\tau = \sum \tau_{j\lambda} \frac{\partial}{\partial z_j}\otimes e_\lambda \in T_X \otimes E$.
We denote it by $\Theta_{E, h} \ge_{\rm Nak} 0$ (resp. $\Theta_{E, h} >_{\rm Nak} 0$).
\end{definition}

If we consider reverse inequalities in the above definition, we can get a definition of Griffiths and Nakano
negativity of the smooth Hermitian vector bundle.

It is known that Griffiths positivity and negativity can be defined without using the curvature current. This characterization and the following
property will be
used in the singular setting.

\begin{proposition} $($\cite[Section 2]{Rau1}$)$.
Let $h$ be a smooth Hermitian metric on $E$. Then the followings are equivalent. \\
(1) $h$ is Griffiths semi-negative. \\
(2) $|u|^2_h$ is plurisubharmonic for every local holomorphic section of $E$. \\
(3) $\log |u|^2_h$ is plurisubharmonic for every local holomorphic section of $E$. \\
(4) $h^\star$ is Griffiths semi-positive $( h^\star$ is the dual metric on the dual bundle $E^\star )$.
\end{proposition}

\subsection{Singular Hermitian metrics on vector bundles}

We adopt the definition of singular Hermitian metrics introduced in \cite{HPS}, \cite{PT}, and \cite{Rau1}.

\begin{definition} $($\cite[Section 3]{BP}, \cite[Definition 17.1]{HPS}, \cite[Definition 2.2.1]{PT} and \cite[Definition 1.1]{Rau1}$)$.
A {\it singular Hermitian metric} $h$ on $E$ is a measurable map from the base manifold $X$ to
the space of non-negative Hermitian forms on the fibers satisfying $0< {\rm det} h < +\infty $ almost everywhere.
\end{definition}

The (Chern) curvature current $\Theta_h$ of a singular Hermitian metric $h$ is locally defined as $\Theta_h = \dbar (h^{-1}\partial h)$.
However, its coefficients are not always measure. This example is found by Raufi in \cite[Theorem 1.5]{Rau1}.
For this reason, we cannot define  the positivity of a singular Hermitian metric by using its curvature current.
The definition of Griffiths positivity and negativity that we will adopt is the following.

\begin{definition}$($\cite[Definition 3.1]{BP}, \cite[Definition 2.2.2]{PT} and \cite[Definition 1.2]{Rau1}$)$\label{GrifPos}
A singular Hermitian metric $h$ is \\
(1) {\it Griffiths semi-negative}, if for any open subset $U\subset X$ and any $u \in H^0(U, E)$,
$\log |u|_h^2$ is plurisubharmonic on $U$. \\
(1') {\it Griffiths semi-negative}, if for any open subset $U\subset X$ and any $u \in H^0(U, E)$,
$|u|_h^2$ is plurisubharmonic on $U$. \\
(2) {\it Griffiths semi-positive}, if the dual metric $h^\star$ is Griffiths semi-negative, or equivalently if
for any open subset $U\subset X$ and any $u \in H^0(U, E^\star)$, $\log |u|_{h^\star}^2$ or $|u|_{h^\star}^2$ is plurisubharmonic on $U$.

\end{definition}

The above definition (1) is equivalent to (1') (cf. \cite[Section 2]{Rau1}).
Nakano semi-negativity of a singular Hermitian metric is also defined without using the curvature current (cf. \cite[Definition 1.8]{Rau1}).
However, the dual bundle of a Nakano negative vector bundle is not necessarily Nakano positive. Hence, the Nakano (semi-)positive singular Hermitian metric has
not been defined. Furthermore, strictly positivity or negativity of a singular Hermitian metric is not generally formulated.
The following result is only known as a definition of strict Griffiths negativity of a singular Hermitian metric.

\begin{definition}$($\cite[Definition 6.1]{Rau1}$)$\label{sGp}
A singular Hermitian metric $h$ on $E$ is strictly {\it Griffiths negative} or {\it $\delta$-negative} if$\colon$ \\
(1) $h$ is Griffiths semi-negative in the sense of Definition \ref{GrifPos}. \\
(2) $F=\{z\in X \colon \det h(z)=0\}$ is a closed set, and there exists an exhaustion of open sets $\{U_j\}_{j=1}^\infty$
such that $\det h>\frac{1}{j}$ on $U_j$, and $\bigcup_{j=1}^\infty U_j = X\setminus F$. \\
(3) There exists some $\delta >0$ such that on $X\setminus F$
$$
\sum_{j, k=1}^n (\Theta_{j\overline{k}}^h s, s)_h\xi_j \overline{\xi}_k \le - \delta \| s \|^2_h ~\omega (\xi, \xi)
$$
in the sense of distributions, for any section $s$ and any vector field $\xi = \sum \xi_j \frac{\partial}{\partial z_j}$ .
Here we let
$$
\sqrt{-1}\Theta_h = \sum_{j, k=1}^n \Theta_{j\overline{k}}^h dz_j\wedge d\overline{z}_k, \hspace{5mm}\Theta_{j\overline{k}}^h \in \mathcal{O}({\rm Hom}(E, E)).
$$
\end{definition}

We say that a singular Hermitian metric $h$ is {\it strictly ($\delta$-)Griffiths positive} if the dual metric is strictly Griffiths negative.
The above condition (2) certifies that the curvature current $\Theta_h$ exists as a current
with measurable coefficients on $X\setminus F$ \cite[Corollary 1.7]{Rau1}.

If $\omega$ is a K\"ahler form, we can define the strict Griffiths negativity and positivity in a more general setting. 

\begin{definition}\label{def:generalnegativity}
Let $\omega$ be a K\"ahler form. 
We say that a singular Hermitian metric $h$ is {\it strictly Griffiths $\delta_\omega$-negative} around $x\in X$ if for any open neighborhood $U$ of $x$ and any K\"ahler potential $\varphi$ of $\omega$ on $U$, i.e. $\sqrt{-1}\ddbar \varphi=\omega$ on $U$, 
$h\cdot e^{-\delta \varphi}$ is Griffiths semi-negative in the sense of Definition \ref{GrifPos}. 
We also say that $h$ is {\it strictly Griffiths $\delta_\omega$-negative} on $X$ if for any point $x\in X$, $h$ is {\it strictly Griffiths $\delta_\omega$-negative} around $x$. 

Similarily, we say that $h$ is {\it strictly Griffiths $\delta_\omega$-positive} if the dual metric $h^\star$ is strictly Griffiths $\delta_\omega$-negative, i.e. $h\cdot e^{\delta \varphi}$ is Griffiths semi-positive in the sense of Definition \ref{GrifPos}. 
\end{definition}

In the above definition, we define the strict negativity without assuming the condition (2) in Definition \ref{sGp}. 
We remark that we can find this $\varphi$ locally by $\ddbar$-lemma and the above definition is independent of the choice of local potentials. 

\section{Some properties of Griffiths positive singular Hermitian metrics}\label{sHmsection}

In this section, we prepare some lemmas and propositions about Griffiths negative or positive singular Hermitian metrics.
Some of them are not directly related to the proof of the main theorems. However, we summarize these properties
in order to improve the outlook.
Throughout this section, let $E$ be a holomorphic vector bundle on the base manifold $X$, and $h$ be
a singular Hermitian metric on $E$.

In order to study analytic properties about the curvature current of $h$, we will take an approximating sequence $\{ h_\nu \}_{\nu =1}^\infty$ of $h$.
Locally, we can take such a sequence, through convolution with an approximate identity, from the following lemma.
\begin{lemma}\cite[Proposition 3.1]{BP}\label{sequence}
Let $h$ be a Griffiths semi-negative singular Hermitian metric. If $E$ is a trivial vector bundle
over a polydisk, then there exists a sequence of smooth Hermitian metrics $\{h_\nu\}_{\nu=1}^\infty$, with negative Griffiths curvature,
decreasing pointwise to $h$ on any smaller polydisk.
\end{lemma}

If $h$ is Griffiths semi-positive, we can also get an increasing approximating sequence by taking the dual metric of above Lemma \ref{sequence}.
Let $\{ h_\nu \}_{\nu =1}^\infty$ be an approximating sequence of $h$. We investigate
analytic properties about $\theta_{h_\nu}$, $\theta_{h}$, $\Theta_{h_\nu}$, and $\Theta_{h}$, where $\theta_{h_\nu}$(resp. $\theta_{h}$)
is the connection form
associated with $h_\nu$ (resp. $h$).
For an arbitrary smooth vector field $\xi$, $\wtilde{\theta}_h$ denotes $\theta_h (\xi)$, and $\wtilde{\Theta}_h$ denotes $\Theta_h(\xi, \xi)$.

\begin{theorem}\cite[Theorem 1.6]{Rau1}\label{negativeweakconv}
Let $X$ be a complex manifold with a positive Hermitian form $\omega$, and let $h$ be a singular Hermitian metric that is Griffiths semi-negative.
Moreover let $\{ h_\nu \}_{\nu =1}^\infty$ be any approximating sequence of smooth Hermitian metrics with Griffiths semi-negative
curvature, decreasing to $h$.

If there exists $\epsilon > 0$ such that $\det h >\epsilon$, then \\
(1) $\wtilde{\theta}_{h_\nu}\in L^2_{loc}(X)$ uniformly in $\nu$, and $\wtilde{\theta}_{h}\in L^2_{loc}(X)$, \\
(2) $\wtilde{\theta}_{h_\nu}\in L^2_{loc}(X)$ weakly converges to $\wtilde{\theta}_{h}$ in $L^2_{loc}(X)$, \\
(3) $\wtilde{\Theta}_{h_\nu}\in L^1_{loc}(X)$ uniformly in $\nu$, $\wtilde{\Theta}_{h}$ has measure coefficients, and $\wtilde{\Theta}_{h_\nu}$ weakly converges to $\wtilde{\Theta}_{h}$ in the sense of measures.
\end{theorem}

The above notations mean that after choosing a basis for $E$ and representing $\wtilde{\theta}_{h_\nu}, \wtilde{\theta}_{h}$, $\wtilde{\Theta}_{h_\nu}$, and $\wtilde{\Theta}_{h}$ as a matrix,
each element of the matrix has measure coefficients, weakly converges, and so on.
If the $L^2$ norm, on a fixed compact subset of $X$, of each element of $\wtilde{\theta}_{h_\nu}$ has an upper bound which
is independent of $\nu$, we say $\wtilde{\theta}_{h_\nu}\in L^2_{loc}(X)$ uniformly in $\nu$.

This type of lemma is only known in the case that the singular Hermitian metric
$h$ is Griffiths semi-negative. We show it in the situation that $h$ is Griffiths semi-positive.
We start by preparing some lemmas for it.

\begin{lemma}\label{seki}
Let $h$ be a Griffiths semi-negative singular Hermitian metric, and $h_{i\overline{j}}$ be the
$(i,j)$ element of $h$. Then it follows that
$$
\partial (h_{i\overline{j}}h_{{k\overline{l}}}) =(\partial h_{i\overline{j}} )h_{k\overline{l}}+
h_{i\overline{j}}(\partial h_{k\overline{l}})
$$
in the sense of distributions. Moreover, we have \\
(1)~$\partial (\det h)\in L^2_{loc}(X)$, \\
(2)~$\partial (\det h_\nu)\in L^2_{loc}(X)$ uniformly in $\nu$, \\
(3)~$\partial (\det h_\nu)\in L^2_{loc}(X)$ weakly converges to $\partial (\det h)\in L^2_{loc}(X)$ in $L^2_{loc}(X)$.

\end{lemma}

\begin{proof}
Since the setting is local, we can assume that $X$ is a polydisk in $\mathbb{C}^n$, $E$ is trivial
over $X$, and $h$ is represented as a matrix. Without any loss of generality, we also can take an approximating sequence $\{ h_\nu \}_{\nu =1}^\infty$ of smooth Hermitian metrics with Griffiths semi-negative
curvature, decreasing to $h$ on $X$. We begin to show that
$$
\int_X  \partial (h_{i\overline{j}}h_{{k\overline{l}}})\chi = \int_X \{ (\partial h_{i\overline{j}} )h_{k\overline{l}}+
h_{i\overline{j}}(\partial h_{k\overline{l}})\} \chi
$$
for any test form $\chi\in C^{\infty (n-1, n)}_{c}(X)$.
It is known that
$$
\int_X  \partial (h_{\nu_{i\overline{j}}}h_{\nu_{{k\overline{l}}}})\chi = \int_X \{ (\partial h_{\nu_{i\overline{j}}} )h_{\nu_{k\overline{l}}}+
h_{\nu_{i\overline{j}}}(\partial h_{\nu_{k\overline{l}}})\} \chi \hspace{8mm} \cdots (\diamondsuit)
$$
for each $\nu$.
Firstly, we have
$$
\left| \int_X \{ \partial (h_{\nu_{i\overline{j}}}h_{\nu_{{k\overline{l}}}})-\partial (h_{{i\overline{j}}}h_{{{k\overline{l}}}}) \} \chi \right|
=\left| \int_X ( h_{\nu_{i\overline{j}}}h_{\nu_{{k\overline{l}}}} - h_{{i\overline{j}}}h_{{{k\overline{l}}}}) ) \partial \chi \right|.
$$
Each element of $h$ and $h_\nu$ is locally bounded uniformly in $\nu$ \cite[Remark 2.2.3]{PT}. Hence
this integral value goes to zero by the Lebesgue convergence theorem.

Secondly, we get
\begin{align*}
& \left| \int_X \{ (\partial h_{\nu_{i\overline{j}}} )h_{\nu_{k\overline{l}}}-(\partial h_{{i\overline{j}}} )h_{{k\overline{l}}} \} \chi \right| \\
& \le \left| \int_X \partial h_{\nu_{i\overline{j}}} ( h_{\nu_{k\overline{l}}}-h_{{k\overline{l}}}) \chi \right| +
\left| \int_X (\partial h_{\nu_{i\overline{j}}}-\partial h_{{i\overline{j}}} )h_{{k\overline{l}}}  \chi \right| \\
&\le \| \partial h_{\nu_{i\overline{j}}} \|_{L^2(K)} \| h_{\nu_{k\overline{l}}}-h_{{k\overline{l}}}\|_{L^2(K)}
+ \left| \int_X (\partial h_{\nu_{i\overline{j}}}-\partial h_{{i\overline{j}}} )h_{{k\overline{l}}}  \chi \right|,
\end{align*}
where $K$ denotes the support of $\chi$.
Locally boundness of $h$ leads to that $\| h_{\nu_{k\overline{l}}}-h_{{k\overline{l}}}\|_{L^2(K)} \to 0$ as $\nu \to \infty$.
Then $\| \partial h_{\nu_{i\overline{j}}} \|_{L^2(K)}$ is uniformly bounded in $\nu$ \cite[Proposition 1.4]{Rau1}. Therefore, the first term goes to zero. The second term also goes to zero
for the reason that
$\partial h_{\nu_{i\overline{j}}}$ weakly converges to $\partial h_{{i\overline{j}}}$ in $L^2(K)$ and $h\chi \in L^2(K)$.
Then we can conclude that $(\partial h_{\nu_{i\overline{j}}} )h_{\nu_{k\overline{l}}}$ weakly converges to
$(\partial h_{{i\overline{j}}} )h_{{k\overline{l}}}$, and $h_{\nu_{i\overline{j}}}(\partial h_{\nu_{k\overline{l}}})$ also
weakly converges to $h_{{i\overline{j}}}(\partial h_{{k\overline{l}}})$.

Finally, taking weak limits of ($\diamondsuit$), we obtain
$$
\partial (h_{i\overline{j}}h_{{k\overline{l}}}) =(\partial h_{i\overline{j}} )h_{k\overline{l}}+
h_{i\overline{j}}(\partial h_{k\overline{l}}).
$$
Repeating this argument, we consequently prove (1), (2), and (3) for the reason that
$\partial h \in L^2_{loc}(X)$, $\partial h_\nu \in L^2_{loc}(X)$ uniformly in $\nu$, and
$\partial h_\nu $ weakly converges to $\partial h $ in $L^2_{loc}(X)$ \cite[Proposition 1.4, Lemma 5.1]{Rau1}.
\end{proof}

Subsequently, we prepare lemmas in the local setting.

\begin{lemma}\label{bunsuu}
Let $h$ be a Griffiths semi-negative singular Hermitian metric. We assume that $\det h >\epsilon$ for some positive constant $\epsilon>0$,
then we have \\
$$
\partial \left( \frac{1}{\det h}\right) = - \frac{\partial \det h}{\det^2 h}
$$
in the sense of distributions.
\end{lemma}

\begin{proof}
It is sufficient to show that
$$
\int_X \partial \left( \frac{1}{\det h} \right) \chi =
\int_X - \frac{\partial \det h}{\det^2 h} \chi
$$
for any test form $\chi\in C^{\infty (n-1, n)}_{c}(X)$. It is known that
$$
\int_X \partial \left( \frac{1}{\det h_\nu} \right) \chi =
\int_X - \frac{\partial \det h_\nu}{\det^2 h_\nu} \chi \hspace{8mm}\cdots (\diamondsuit)
$$
for each $\nu$.

Firstly, it follows that $\partial \left( \frac{1}{\det h_\nu} \right)$ weakly converges to
$\partial \left( \frac{1}{\det h} \right)$ for the reason that
$\frac{1}{\det h_\nu}$ is increasing to $\frac{1}{\det h}$.

Secondly, we begin to show that
$\frac{\partial \det h_\nu}{\det^2 h_\nu}$ weakly converges to $\frac{\partial \det h}{\det^2 h}$.
We have
\begin{align*}
&\left| \int_X \frac{\partial \det h_\nu}{\det^2 h_\nu} \chi - \int_X \frac{\partial \det h}{\det^2 h} \chi \right| \\
&\le \left| \int_X \left( \frac{1}{\det^2 h_\nu}- \frac{1}{\det^2 h}\right) \partial \det h_\nu \chi \right| +
\left| \int_X \frac{1}{\det^2 h}\left( \partial \det h_\nu - \partial \det h \right) \chi \right| \\
&\le C \Bigl\| \frac{1}{\det^2 h_\nu}- \frac{1}{\det^2 h}\Bigr\|_{L^2(K)} \|\partial \det h_\nu\|_{L^2(K)}
+ \left| \int_X \frac{1}{\det^2 h}\left( \partial \det h_\nu - \partial \det h \right) \chi \right|,
\end{align*}
where $C$ denotes the supremum of $\chi$ on $X$, and $K$ denotes a support of $\chi$.
The first term goes to zero as $\nu \to \infty$ for the reason that $\|\partial \det h_\nu\|_{L^2(K)}$ uniformly in $\nu$ and
$\frac{1}{\det^2 h_\nu}$ is increasing to $\frac{1}{\det^2 h}<\frac{1}{\epsilon^2}$. For the second term, we know that
$\partial \det h_\nu$ weakly converges $\partial \det h$ in $L^2(K)$ by the Lemma \ref{seki} and $\frac{1}{\det^2 h}<\frac{1}{\epsilon^2}$ is of course $L^2(K)$ function.
Hence it goes to zero as $\nu \to \infty$.

Finally, we can conclude that $\frac{\partial \det h_\nu}{\det^2 h_\nu}$ weakly converges to $\frac{\partial \det h}{\det^2 h}$.
Taking weak limits of ($\diamondsuit$),
we obtain
$$
\partial \left( \frac{1}{\det h} \right) = - \frac{\partial \det h}{\det^2 h}
$$
in the sense of the distributions.
\end{proof}

\begin{lemma}\label{diff}
Let $h$ be a Griffiths semi-negative singular Hermitian metric, and $\what{h}$ be the adjugate matrix of $h$. We assume that $\det h >\epsilon$ for some positive constant $\epsilon>0$,
then we have \\
$$
\partial \left( \frac{1}{\det h}h\right) = \partial \left( \frac{1}{\det h}\right)h + \frac{1}{\det h}\partial h ,
~\partial \left( \frac{1}{\det h}\what{h}\right) = \partial \left( \frac{1}{\det h}\right)\what{h} + \frac{1}{\det h}\partial \what{h}
$$
in the sense of distributions.
\end{lemma}

\begin{proof}
The proof of the first part is almost the same as the second part. It is enough to
show that
$$
\int_X \partial \left( \frac{1}{\det h}\what{h}\right) \chi =
\int_X \Bigl\{ \partial \left( \frac{1}{\det h}\right)\what{h} + \frac{1}{\det h}\partial \what{h} \Bigr\} \chi
$$
for any test form $\chi\in C^{\infty (n-1, n)}_{c}(X)$.
It is known that
$$
\int_X \partial \left( \frac{1}{\det h_\nu}\what{h}_\nu\right) \chi =
\int_X \Bigl\{ \partial \left( \frac{1}{\det h_\nu}\right)\what{h}_\nu + \frac{1}{\det h_\nu}\partial \what{h}_\nu \Bigr\} \chi
\hspace{8mm}\cdots (\diamondsuit)
$$
for each $\nu$.
For the left-hand side of the above equation, we will show that
$\partial \left( \frac{1}{\det h_\nu}\what{h}_\nu\right)$ weakly converges to
$\partial \left( \frac{1}{\det h}\what{h}\right)$. We have
$$
\left| \int_X \Bigl\{ \partial \left( \frac{\what{h}_\nu}{\det h_\nu}\right)-\partial \left( \frac{\what{h}}{\det h}\right) \Bigr\} \chi \right|
= \left| \int_K   \left( \frac{\what{h}_\nu}{\det h_\nu}- \frac{\what{h}}{\det h}\right)  \partial \chi \right|.
$$
Here
$$
\left| \left( \frac{\what{h}_\nu}{\det h_\nu}-\frac{\what{h}}{\det h}\right) \partial \chi \right|
\le C'\left( \frac{|\what{h}_\nu|}{\det h_\nu}+\frac{|\what{h}|}{\det h}\right) \le \frac{2(r-1)!C^{r-1}C'}{\epsilon} \in  L^1(K),
$$
where $K$ denotes a support of $\chi$, $C$ denotes the supremum of $h_0$ on $K$,
and $C'$ denotes the supremum of $\partial \chi$ on $K$.
The constant $C$ satisfies the following inequalities that $| h_\nu | \le C$ and $| h | \le C$ on $K$ \cite[Remark 2.2.3]{PT}.
Then we conclude that $\partial \left( \frac{1}{\det h_\nu}\what{h}_\nu\right)$ weakly converges to
$\partial \left( \frac{1}{\det h}\what{h}\right)$ by the Lebesgue convergence theorem.

For the right-hand side of the above equation,
we will prove that $\partial \left( \frac{1}{\det h_\nu}\right)\what{h}_\nu$ weakly converges to $\partial \left( \frac{1}{\det h}\right)\what{h}$ and
$\frac{1}{\det h_\nu}\partial \what{h}_\nu$ weakly converges to $\frac{1}{\det h}\partial \what{h}$.

Firstly, we have
\begin{align*}
&\left| \int_X \Bigl\{ \partial \left( \frac{1}{\det h_\nu}\right)\what{h}_\nu-\partial \left( \frac{1}{\det h}\right)\what{h} \Bigr\} \chi \right| \\
&= \left| \int_X \left( \frac{\partial \det h}{\det^2 h}\what{h} -\frac{\partial \det h_\nu}{\det^2 h_\nu} \what{h}_\nu \right)  \chi \right| \\
&\le \left| \int_X  \left( \partial \det h - \partial \det h_\nu \right)\frac{\what{h}}{\det^2 h} \chi \right| +
\left| \int_X  \partial \det h_\nu \left( \frac{\what{h}}{\det^2 h} - \frac{\what{h}_\nu}{\det^2 h_\nu} \right) \chi \right| \\
&\le \left| \int_X  \left( \partial \det h - \partial \det h_\nu \right)\frac{\what{h}}{\det^2 h} \chi \right| +
C'' \| \partial \det h_\nu \|_{L^2(K)} \Bigl\| \frac{\what{h}}{\det^2 h} - \frac{\what{h}_\nu}{\det^2 h_\nu} \Bigr\|_{L^2(K)}
\end{align*}
by Lemma \ref{bunsuu}, where $K$ denotes a support of $\chi$, and $C''$ denotes the supremum of $\chi$ on $K$.
The first term goes to zero for the reason why $\partial \det h_\nu$ weakly converges to $\partial \det h$ in $L^2(K)$ from the results of Lemma \ref{seki} and
$\frac{\hat{h}}{\det^2 h}\in L^2(K)$.
For the second term, we have $\| \partial \det h_\nu \|_{L^2(K)} $ uniformly in $\nu$, and
$$
\left| \frac{\what{h}}{\det^2 h} - \frac{\what{h}_\nu}{\det^2 h_\nu} \right|^2 \le \frac{4C^2}{\epsilon^4}\in L^1(K).
$$
Therefore, it goes to zero by the Lebesgue convergence theorem.

Finally, taking weak limits of $(\diamondsuit)$, we have
$$
\partial \left( \frac{1}{\det h}\what{h}\right)=
\partial \left( \frac{1}{\det h}\right)\what{h} + \frac{1}{\det h}\partial \what{h}
$$
in the sense of distributions.
\end{proof}

Using the above lemmas, we can get the positive version of Theorem \ref{negativeweakconv}. More precisely,
we have the following theorem.

\begin{theorem}\label{positiveweakconv}
Let $X$ be a complex manifold with a positive Hermitian form $\omega$, and let $h$ be a singular Hermitian metric that is Griffiths semi-positive.
Moreover let $\{ h_\nu \}_{\nu =1}^\infty$ be any approximating sequence of smooth Hermitian metrics with Griffiths semi-positive
curvature, increasing to $h$.

If there exists $C > 0$ such that $\det h < C$, then \\
(1) $\wtilde{\theta}_{h_\nu}\in L^2_{loc}(X)$ uniformly in $\nu$, and $\wtilde{\theta}_{h}\in L^2_{loc}(X)$, \\
(2) $\wtilde{\theta}_{h_\nu}$ weakly converges to $\wtilde{\theta}_{h}$ in $L^2_{loc}(X)$, and \\
(3) $\wtilde{\Theta}_{h_\nu}\in L^1_{loc}(X)$ uniformly in $\nu$, $\wtilde{\Theta}_{h}$ has measure coefficients, and $\wtilde{\Theta}_{h_\nu}$ weakly converges to $\tilde{\Theta}_{h}$ in the sense of measures.
\end{theorem}

\begin{proof}
First of all, the dual metric $h^\star$ satisfies the assumption of the above lemmas.
Repeating the proof of Lemma \ref{diff}, we get
$$
\partial \Bigl(\frac{1}{\det h^\star}\what{h}^{\star}h^\star\Bigr) = \Bigl(\partial (\frac{1}{\det h^\star})\Bigr)\what{h}^{\star}h^\star + \frac{1}{\det h^\star}\partial (\what{h}^{\star}h^\star ).
$$
Using Lemma \ref{seki}, we obtain
\begin{align*}
0&= \partial (h^{\star -1}h^\star)\\
&= \partial \Bigl(\frac{1}{\det h^\star}\what{h}^{\star}h^\star\Bigr)\\
&=\Bigl(\partial (\frac{1}{\det h^\star})\Bigr)\what{h}^{\star}h^\star + \frac{1}{\det h^\star}(\partial \what{h}^{\star})h^\star + \frac{1}{\det h^\star}\what{h}^{\star}(\partial h^\star)\\
&= (\partial h^{\star -1} )h^\star + h^{\star -1} (\partial h^\star )
\end{align*}
in the sense of distributions,
hence
we have $\theta_{h}=-{}^t \theta_{h^\star}$. For each $\nu$, we also have
$\theta_{h_\nu}=-{}^t \theta_{h^\star_\nu}$. Using Theorem \ref{negativeweakconv},
we can prove the part (1) and (2). Part (3) also follows since
$$
\Theta_h = \overline{\partial}\theta_h = -{}^t \overline{\partial}\theta_{h^\star}= -{}^t \Theta_{h^\star}.
$$
\end{proof}

\begin{remark}
If $h$ is smooth, the above equation $\Theta_h = - {}^t\Theta_{h^\star}$ is a well-known fact.
However, if $h$ is singular, differential is in the sense of distributions. Hence $0=\partial h^{-1}h + h^{-1}\partial h$ makes no sense.
For this reason, we do not know whether the equation $\Theta_h = - {}^t\Theta_{h^\star}$ holds when $h$ is singular.
\end{remark}

\section{Proof of Theorem \ref{mainthm1}}\label{proofthm1}
In this section, we will prove Theorem \ref{mainthm1}. First of all, we need some lemmas. Throughout section \ref{proofthm1} and \ref{proofthm2},
$X$ denotes an $n$-dimensional projective manifold,
$\omega$ denotes a K\"ahler form on $X$,
$E$ denotes a holomorphic vector bundle over of rank $r$,
and $h$ denotes a Hermitian metric on $E$.
Let $L^2_{(p,q)}(X, E, h, \omega)$ (resp. $L^2_{loc(p,q)}(X, E, h, \omega)$) be the space of square integrable (resp. locally square integrable)
$E$-valued $(p, q)$-forms on $X$.

\begin{lemma}\label{smoothdbar}
Let $h'$ be a smooth and strictly Griffiths $\delta_\omega$-positive metric on $E$. For any $u\in L^2_{(n, n)}(X, E, h')$, there exists $g\in L^2_{(n, n-1)}(X, E, h')$ such that
$$
\dbar g =u, \hspace{5mm} \| g\|^2_{L^2} \le \frac{1}{\delta n} \| u \|^2_{L^2}.
$$
\end{lemma}

\begin{proof}
We will compute the Hermitian operator $\left[ \sqrt{-1}\Theta_{h'} ,\Lambda \right]$, where
$\Lambda $ is the adjoint operators of $L$ which is defined by $Lu = \omega \wedge u$, and $\left[ ~,~ \right]$ is graded Lie bracket.
Writing
\begin{align*}
\sqrt{-1}\Theta_{h'}=\sqrt{-1}&\sum_{1\le j, k\le n, 1\le \lambda, \mu \le r} c_{j \overline{k} \lambda \mu}dz_j\wedge d\overline{z}_k \otimes e_\lambda^\star \otimes e_\mu, \\
\omega =\sqrt{-1}&\sum_{1\le j\le n} dz_j \wedge d\overline{z}_j
\end{align*}
at a fixed point $p\in X$ as in Section \ref{prelim}, we compute the operator
$
\left[ \sqrt{-1}\Theta_{h'} ,\Lambda \right].
$
For any $E$-valued $(n,n)$-$L^2$ form $u=\sum_{\lambda=1}^r u_\lambda dz_1\wedge \ldots \wedge d\overline{z}_n\otimes e_\lambda$, we get
$$
\left[ \sqrt{-1}\Theta_{h'} ,\Lambda \right]u=\sqrt{-1}\Theta_{h'} \Lambda u = \sum_{1\le j \le n, 1\le \lambda, \mu \le r} c_{j\overline{j}\lambda \mu} u_\lambda dz_1\wedge \ldots \wedge d\overline{z}_n \otimes e_\mu .
$$
Therefore, the following equations hold
\begin{align*}
&( \left[ \sqrt{-1}\Theta_{h'} ,\Lambda \right] u, u ) _{h'} \\
&= ( \sum_{1\le j\le n, 1\le \lambda ,\mu\le r} c_{j\overline{j}\lambda \mu} u_\lambda dz_1\wedge \ldots \wedge d\overline{z}_n\otimes e_\mu , \sum_{1\le \nu\le r} u_\nu dz_1\wedge \ldots \wedge d\overline{z}_n\otimes e_\nu ) _{h'} \\
&=\sum_{1\le j\le n, 1\le \lambda ,\mu\le r} c_{j\overline{j}\lambda \mu}u_\lambda \overline{u}_\mu \\
&= \sum_{1\le j\le n} \Bigl(\Theta_{j\overline{j}}^{h'} (\sum_{\lambda =1}^r u_\lambda), (\sum_{\lambda =1}^r u_\lambda)\Bigr)_{h'}.
\end{align*}
For each $j$, by taking a vector field $\xi = \frac{\partial}{\partial z_j}$ we show that
$$
\Bigl(\Theta_{j\overline{j}}^{h'} (\sum_{\lambda =1}^r u_\lambda), (\sum_{\lambda =1}^r u_\lambda)\Bigr)_{h'} \ge \delta | u |^2_{h'}
$$
from the definition of a strictly Griffiths $\delta$-positive metric. Then we get
$$
\sum_{1\le j\le n} \Bigl(\Theta_{j\overline{j}}^{h'} (\sum_{\lambda =1}^r u_\lambda), (\sum_{\lambda =1}^r u_\lambda)\Bigr)_{h'} \ge \delta n | u|^2_{h'},
$$
and we see that the operator $\left[ \sqrt{-1}\Theta_{h'} ,\Lambda \right]$ is positive definite.
Hence it follows that
\begin{align*}
(\left[ \sqrt{-1}\Theta_{h'} ,\Lambda \right]^{-1}u,u)_{h'}&\le \frac{1}{\delta n}(u,u)_{h'}, \\
\int_X (\left[ \sqrt{-1}\Theta_{h'} ,\Lambda \right]^{-1}u,u)_{h'} dV_{\omega} &\le \frac{1}{\delta n}\| u\|^2_{L^2} <+\infty.
\end{align*}
We then can conclude that there exists $g\in L^2_{(n, n-1)}(X, E, h')$ such that $\dbar g =u$ and  $\| g\|^2_{L^2} \le \frac{1}{\delta n} \| u \|^2_{L^2}$, thanks to H\"ormander's $L^2$ estimate (cf. \cite[Theorem 5.1]{DemAna}).
\end{proof}

We prove one more lemma. It is a generalization of the argument of \cite[Section 4]{Rau2}.
\begin{lemma}\label{convolution}
Let $Z$ be a $n$-dimensional submanifold of  $\mathbb{C}^N$, $U$ be an open neighborhood of $Z$ in $\mathbb{C}^N$, and $p\colon U\rightarrow Z$ be a holomorphic retraction map such that
$p\circ i= {\rm id}_Z$, where $i$ is an inclusion map $i\colon Z\rightarrow U$.
We also fix a K\"ahler metric $\omega$ such that $\omega$ is $\ddbar$-exact on $Z$. 
Assume that there is a trivial vector bundle $E$ over $Z$ equipped with a singular Hermitian metric $h$ which is strictly Griffiths $\delta_\omega$-positive in the sense of Definition \ref{def:generalnegativity}. 

Then we can take a sequence of smooth Hermitian metrics $\{ h_\nu \}_{\nu=1}^\infty $
approximating $h$ on any relatively compact subset of $Z$ such that $h_\nu$ is strictly Griffiths $\delta_\omega$-positive for each $\nu$. 
\end{lemma}

\begin{proof}
The assumption implies that there exists a smooth K\"ahler potential  $\varphi$ of $\omega$ on $Z$ such that $\sqrt{-1} \ddbar \varphi =\omega$ and $h\cdot e^{\delta \varphi}$ is Griffiths semi-positive. 
We see that $p^\star(h\cdot e^{\delta \varphi})$ is also Griffiths semi-positive from \cite[Lemma 2.3.2]{PT}. 

Then we take an approximate identity, i.e. $\chi \in C^\infty_c (U)$, $\chi \ge 0$,
$\chi(z)=\chi (|z|), \int_{\mathbb{C}^n}\chi dV =1$, and $\chi_\nu (z)= \nu^n \chi(\nu z)$. 
Since $p^\star E$ is also trivial over $U$, 
we consider smooth Griffiths semi-positive Hermitian metrics $((p^\star (h\cdot e^{\delta \varphi}))^\star \ast \chi_\nu )^\star $ on $p^\star E$ increasing pointwise to $p^\star (h\cdot e^{\delta \varphi})$  on any relatively compact subset of $U$. 
Set $g_\nu:= i^\star (((p^\star (h\cdot e^{\delta \varphi}))^\star \ast \chi_\nu )^\star)$ and $h_\nu:=g_\nu e^{-\delta \varphi}$. 
Since $i$ is an inclusion map, we see that $\{ g_\nu\}$ is a sequence of smooth Hermitian metrics on $E$ with semi-positive Griffiths curvature, increasing pointwise to $h\cdot e^{\delta \varphi}$ on any relatively compact subset of $Z$. 
Hence, $h_\nu$ is strictly Griffiths $\delta_\omega$-positive and increasing to $h$ on any relatively compact subset of $Z$. 
\end{proof}

Then we will prove Theorem \ref{mainthm1}.

\begin{proof}[\indent\sc Proof of Theorem \ref{mainthm1}]
By Serre's GAGA, there exists a Zariski open subset $Z\neq \emptyset$ such that $Z\subset X\setminus S$, $E|_Z$ is a trivial over $Z$, and $\omega$ is $\ddbar$-exact on $Z$ 
for the reason that $X$ is a projective manifold.
We can take $Z$ as a Stein open subset. Then $Z$ can be properly imbedded in $\mathbb{C}^N$ for some large $N$.
We regard $Z$ as a submanifold of $\mathbb{C}^N$.
From Siu's result in  \cite{Siu}, there exists an open neighborhood $U$ of $Z$ in $\mathbb{C}^N$ which is a holomorphic retract of $Z$. Let $p\colon U\rightarrow Z$ be a holomorphic retraction map such that
$p\circ i= {\rm id}_Z$, where $i$ is an inclusion map $i\colon Z\rightarrow U$. Since $E|_Z$ is a trivial vector bundle, $p^\star E$ is also trivial on $U$.
We can take an exhaustion $\{ Z_j\}_{j=1}^\infty$ of $Z$, where each $Z_j$ is a relatively compact Stein subdomain.
From the results of Lemma \ref{convolution}, we can take a sequence of smooth Hermitian metrics $\{ h_\nu\}_{\nu =1}^\infty$ with strictly Griffiths $\delta_\omega$-positive curvature, increasing pointwise to $h$ on $Z_j$ for all $j$. 

For fixed $j$, we get the following inequality
$$
\int_{Z_j} |f|_{h_\nu, \omega}^2 dV_\omega \le \int_{Z_j} |f|_{h, \omega}^2 dV_\omega \le \int_{X} |f|_{h, \omega}^2 dV_\omega <+\infty
$$
for the reason that $\{ h_\nu \}_{\nu=1}^\infty$ is increasing to $h$ and $f\in L^2_{(n, n)}(X, E, h, \omega)$.
Applying Lemma \ref{smoothdbar} and Lemma \ref{convolution} on $Z_j$, we get
$E$-valued $(n,n-1)$-$L^2$ form $g_\nu$ on $Z_j$ such that $\dbar g_\nu=f$, and
$$
\hspace{5mm}\int_{Z_j} | g_\nu |^2_{h_\nu,\omega} dV_\omega \le \frac{1}{\delta n}\int_{Z_j} | f|^2_{h_\nu,\omega} dV_\omega
\le \frac{1}{\delta n}\int_{Z_j} | f|^2_{h,\omega} dV_\omega < +\infty ,
$$
since Lemma \ref{smoothdbar} also holds for Stein manifolds.
Moreover, the right-hand side of above inequalities has upper bound independent of $\nu$. Therefore, we can find a weakly convergent subsequence $\{ g_{\nu_k} \}_{k=1}^\infty$ by using a diagonal argument and monotonicity of $\{ h_\nu \}_{\nu=1}^\infty $.
It follows that $\{ g_{\nu_k} \}_{k=1}^\infty$ weakly converges in $L^2_{(n, n)}(Z_j, E, h_\nu, \omega)$ and the weak limit $g_j$ is in $L^2_{(n,n)}(Z_j, E, h, \omega)$,
i.e.
$$
\int_{Z_j} | g_j |^2_{h,\omega} dV_\omega \le \frac{1}{\delta n}\int_{Z_j} | f|^2_{h,\omega} dV_\omega
\le \frac{1}{\delta n}\int_{Z} | f|^2_{h,\omega} dV_\omega ,
$$
equivalently,
$$
\delta n \int_{Z_j} | g_j |^2_{h,\omega} dV_\omega \le \int_{Z} | f|^2_{h,\omega} dV_\omega .
$$
The right-hand side of the above inequality is independent of $j$, repeating a diagonal argument and taking weak limits, 
then we obtain $E$-valued $(n,n-1)$-$L^2$ form $g$ on $Z$ such that
$$
\int_{Z} | g |^2_{h,\omega} dV_\omega \le \frac{1}{\delta n}\int_{Z} | f|^2_{h,\omega} dV_\omega .
$$
Letting $g$ be $0$ on $X\setminus Z$, $g$ is in $L^2_{(n, n-1)}(X, E, h, \omega)$, and we get
$$
\dbar g =f, \int_{X} | g |^2_{h,\omega} dV_\omega \le \frac{1}{\delta n}\int_{X} | f|^2_{h,\omega} dV_\omega
$$
on $X$ from the following lemma.
\end{proof}

\begin{lemma}$($\cite[Lecture 5]{B}$)$\label{distribution}
Let $X$ be a complex manifold and let $S$ be a complex hypersurface in $X$. Let $f$ and $g$ be (possibly bundle valued) forms with $L^2_{loc}$
coefficients on $X$ satisfying $\dbar g =f$ on $X\setminus S$. Then the equation $\dbar g =f$ also holds on $X$ in the sense
of distributions.
\end{lemma}

\begin{remark}\label{integral}
We only have to show $\dbar g =f $ on a local open set $U\subset X$, where $U$ contains points of $X\setminus Z$ and $E|_U$ is trivial on $U$. Since $h$ is Griffiths semi-positive, there exists a sequence of smooth Hermitian metrics $\{ h_\nu \}_{\nu =1}^\infty$ increasing to $h$ in this local setting.
For this reason, if $f$ and $g$ are $L^2_{loc}$ with respect to the singular Hermitian metric $h$, then $f$ and $g$ are $L^2_{loc}$ with respect to
the smooth Hermitian metric $h_0$. Therefore we can apply Lemma \ref{distribution} to Theorem \ref{mainthm1}.
\end{remark}

Before we prove Corollary \ref{maincor1}, we introduce the definition of the Lelong number of a singular Hermitian metric on a line bundle.
\begin{definition}\label{linelelong}
Let $L\to X $ be a line bundle over $X$, and
$g$ be a singular Hermitian metric on $L$. \\
If $g$ is semi-negative,
the Lelong number of $g$ at $x$ is
defined as
$$
\nu(\log g, x)=\liminf_{z \to x} \frac{\log g(z)}{\log |z-x|} .
$$
If $g$ is semi-positive,
the Lelong number of $g$ at $x$ is
defined as
$$
\nu(-\log g, x)=\liminf_{z \to x} \frac{-\log g(z)}{\log |z-x|} .
$$
Here $x$ is a point of $X$, $z$ is a coordinate around $x$.
Taking a basis for $L$, we represent $g$ as a positive function.
\end{definition}

The above definition is independent of the choice of local coordinates.
If $g$ is semi-negative (resp. positive), $\log g$ (resp. $-\log g$) is locally plurisubharmonic.
Therefore Definition \ref{linelelong} coincide with the usual definition of the Lelong number of a closed
positive current (cf. \cite[Theorem 2.8]{DemAna}).

Using these notions, we will prove Corollary \ref{maincor1}.
We recall that $\det h$ is a semi-positive (resp. negative) if $h$ is a Griffiths semi-positive (resp. negative)
singular Hermitian metric (cf. \cite[Proposition 25.1]{HPS}, \cite[Proposition 1.3]{Rau1}).

\begin{proof}[\indent\sc Proof of Corollary \ref{maincor1}]
Let $C^\infty_{(n,n)}(X, E)$ be the space of smooth $E$-valued $(n,n)$ forms on $X$ 
and $\mathcal{U}=\{ U_i\}_{i\in I}$ be a locally finite open cover of $X$ such that $U_i$ are biholomorphic to a polydisc.
For the reason that the Lelong number of $\frac{1}{2}\log \det h^\star $ is less than $1$ for a point $x\in X$, we have
$$
e^{-\log \det h^\star}= \frac{1}{\det h^\star}\in L^1_{loc}(X)
$$
from the results of Skoda \cite{Sko}. Then $|s|^2_h$ is an $L^1_{loc}$ form for any $s\in  C^\infty_{(n,n)}(X, E)$ since $h=\frac{1}{\det h^\star}\what{h}^\star$ and each element of $\what{h}^\star$
is locally bounded \cite[Remark 2.2.3]{PT}. Here $\what{h}^\star$ is the adjugate matrix of $h^\star$.
Hence, there is an inclusion map
$$
C^\infty_{(n,n)}(X, E) \hookrightarrow L^2_{loc(n,n)}(X, E, h, \omega).
$$
We know that $U_{i_0}\cap \cdots \cap U_{i_l}$ is a pseudoconvex domain for all $\{ i_0, \cdots, i_l\} \subset I$. 
Taking a sequence of smooth Hermitian metrics $\{ h_\nu\}_{\nu=1}^\infty$ approximating $h$ and repeating the argument of the proof of Theorem \ref{mainthm1}, 
we can solve the $\overline{\partial}$-equation on $U_{i_0}\cap \cdots \cap U_{i_l}$ with respect to $h$. 
Hence, 
we have the isomorphism
$$
H^n(X, K_X \otimes E)
$$
$$
\cong \frac{\{ f\in L^2_{loc(n,n)}(X, E, h, \omega)\}}{\{ h\in L^2_{loc(n,n)}(X, E, h, \omega)\colon \dbar g=h, g \in L^2_{loc(n,n-1)}(X, E, h, \omega)\}}
$$
from the results of sheaf cohomology. 
This is a singular version of isomorphism theorems (cf.~\cite[Proposition 3.1]{Oh}).
For any $f\in L^2_{loc(n,q)}(X, E, h, \omega)$,
we can obtain an $E$-valued $(n,n-1)$-$L^2$ form $g$ on $X$ such that $\dbar g =f$ and $g \in L^2_{(n, n-1)}(X, E, h, \omega)$ by using Theorem \ref{mainthm1}. We can conclude that
$
H^n(X, K_X\otimes E)=0.
$
\end{proof}

\section{Proof of Theorem \ref{mainthm2}}\label{proofthm2}
We will prove Theorem \ref{mainthm2}. The proof of Theorem \ref{mainthm2} is same as that of Theorem \ref{mainthm1}.
To begin with, we prepare a lemma with respect to a smooth Hermitian metric.
\begin{lemma}\label{nakano}
Let $h'$ be a smooth Hermitian metric that is strictly $\delta$-Griffiths positive.
Then the metric $h' \otimes \det h'$ is strictly $\delta r$-Nakano positive, i.e. the inequality
$$
\Theta_{h'\otimes \det h'}(\tau, \tau) \ge \delta r |\tau |^2_{h'\otimes \det h', \omega}
$$
holds for all non-zero tensors $\tau \in T_X \otimes E\otimes \det E$.

\end{lemma}

\begin{proof}
We write
$$
\sqrt{-1}\Theta_{h'}=\sqrt{-1}\sum_{1\le j, k\le n, 1\le \lambda, \mu \le r} c_{j \overline{k} \lambda \mu}dz_j\wedge d\overline{z}_k \otimes e_\lambda^\star \otimes e_\mu
$$
as in Section \ref{prelim}, where $(e_1, \ldots ,e_r)$ is orthonormal with respect to $h'$ at a fixed point.
It follows that
$$
\Theta_{E\otimes \det E}= \Theta_{E} + {\rm Tr}~\Theta_{E} \otimes h' .
$$
Thus we should prove the following inequality
$$
(\Theta_{E} + {\rm Tr}~\Theta_{E} \otimes h' )(u,u) \ge \delta r |u|^2_{h',\omega}
$$
for any section $u= \sum u_{j\lambda} \frac{\partial}{\partial z_j}\otimes e_\lambda \in T_X \otimes E$.
We have the inequality
$$
(\Theta_{E} + {\rm Tr}~\Theta_{E} \otimes h' )(u,u) \ge \sum_{1\le j,k \le n, 1\le \lambda \le r} c_{j\overline{k}\lambda \lambda}u_{j\lambda}\overline{u}_{k\lambda}
$$
from the result of \cite[Proposition 10.14]{DemAna}. Computing the right-hand side of the above inequality, we obtain
$$
\sum_{1\le j,k \le n, 1\le \lambda \le r} c_{j\overline{k}\lambda \lambda}u_{j\lambda}\overline{u}_{k\lambda} = \sum_{1\le j,k\le n, 1\le \lambda \le r} (\Theta_{j\overline{k}}^{h'}e_\lambda , e_\lambda )_{h'}u_{j\lambda}\overline{u}_{k\lambda}
\ge \delta r |u|^2_{h',\omega}
$$
for the reason that $h'$ is $\delta$-Griffiths positive.
\end{proof}
Then we will prove Theorem \ref{mainthm2}.

\begin{proof}[\indent\sc Proof of Theorem \ref{mainthm2}]
We take $Z, \{Z_j\}_{j=1}^\infty, U$, and $p$ as in the proof of Theorem \ref{mainthm1}. Thanks to Lemma \ref{convolution}, we can take a sequence of smooth Hermitian metrics $\{ h_\nu \}_{\nu=1}^\infty $ with strictly Griffits $\delta_\omega $-positive curvature, 
approximating $h$ on any relatively compact subset of $Z$. 
From the above Lemma \ref{nakano}, the Hermitian metric $h_\nu \otimes \det h_\nu$ is $\delta r$ Nakano positive. Then for each $\nu$, there exists an $E\otimes \det E$-valued $(n,q-1)$-$L^2$ form $g_\nu$ on fixed $Z_j$ such that $\dbar g_\nu =f $, and
\begin{align*}
\int_{Z_j} | g_\nu |^2_{h_\nu \tensor h_\nu,\omega} dV_\omega &\le \frac{1}{\delta qr}\int_{Z_j} | f|^2_{h_\nu \tensor h_\nu,\omega} dV_\omega \\
&\le \frac{1}{\delta qr}\int_{Z_j} | f|^2_{h\tensor h,\omega} dV_\omega < +\infty
\end{align*}
for the reason that $\{ h_\nu \tensor h_\nu\}_{\nu=1}^\infty$ is also increasing pointwise to $h \tensor h$ on any relatively compact subset of $Z$ and we can apply H\"ormander's $L^2$ estimates to it (cf. \cite[Chapter VIII, Theorem 6.1]{DemCom}, \cite{Rau2}).
Taking weak limits $\nu \to \infty$ as in the proof of Theorem \ref{mainthm1}, we can take an $E\otimes {\rm det}E$-valued $(n,q-1)$-$L^2$ form $g_j$ on fixed $Z_j$ such that
$$
\int_{Z_j} | g_j |^2_{h\tensor h,\omega} dV_\omega \le \frac{1}{\delta qr}\int_{Z_j} | f|^2_{h\tensor h,\omega} dV_\omega
\le \frac{1}{\delta qr}\int_{Z} | f|^2_{h\tensor h,\omega} dV_\omega .
$$
Moreover taking limits $j\rightarrow \infty$ as in the proof of Theorem \ref{mainthm1}, we obtain an $E\otimes \det E$-valued $(n,q-1)$-$L^2$ form $g$ on $Z$ such that
$$
\int_{Z} | g |^2_{h\tensor h,\omega} dV_\omega \le \frac{1}{\delta qr}\int_{Z} | f|^2_{h\tensor h,\omega} dV_\omega .
$$
Letting $g$ be 0 on $X \setminus Z$, we see that $g$ is in $L^2_{(n, q-1)}(X, E\tensor E, h\tensor h, \omega)$.
Then we get $\dbar g =f$ on $X$, and
$$
\int_{X} | g |^2_{h\tensor h,\omega} dV_\omega \le \frac{1}{\delta qr}\int_{X} | f|^2_{h\tensor h,\omega} dV_\omega .
$$
\end{proof}
Theorem \ref{mainthm2} lead to Corollary \ref{maincor2}.
The proof of it is the same as the one of Corollary \ref{maincor1}.

\begin{proof}[\indent\sc Proof of Corollary \ref{maincor2}]
Locally, we have
$$
h\tensor h = \frac{1}{\det^2 h^\star} \widehat{h}^\star,
$$
where $\widehat{h}^\star$ is the adjugate matrix of $h^\star$. From the results of Skoda \cite{Sko},
we have
$$
\frac{1}{\det^2 h^\star} \in L^1_{loc}(X) .
$$
Repeating the argument of the proof of Corollary \ref{maincor1}, we can conclude that
$$
H^q(X, K_X\otimes E\tensor E)=0
$$
for $q>0$.
\end{proof}

We have an application of Corollary \ref{maincor2}. We show the following example.

\begin{example}
Let $V$ be a complex vector space of dimension $n+1$, and $X=P(V):=(V\setminus \{ 0 \})/\mathbb{C}^\star = \mathbb{P}^n$
be the projective space of $V$.
We also let $\mathcal{O}(-1)$ denote the
tautological line bundle, and $Q$ denote the
quotient bundle $V/\mathcal{O}(-1)$. 
Then there do not exist any Griffiths semi-positive singular Hermitian metrics 
satisfying strict Griffiths $\delta_\omega$-positivity outside some proper analytic subset $S$ and $\nu(- \log \det h, x)<1$ for all $x\in X$ on $Q$. 
\end{example}

\begin{proof}
There exists an exact sequence
$$
0\longrightarrow \mathcal{O}(-1) \longrightarrow V \longrightarrow Q\longrightarrow 0
$$
over $X=\mathbb{P}^n$.
We have the isomorphisms
$$
\det Q \cong \mathcal{O}(1),
$$
$$
{\rm T}\mathbb{P}^n =Q \otimes \mathcal{O}(1) \cong Q \tensor Q
$$
from \cite[Chapter VII, Example 8.4]{DemCom}. If $Q$ has a Griffiths semi-positive singular Hermitian metric satisfying strict Griffiths $\delta_\omega$-positivity outside some proper analytic subset $S$ and $\nu(- \log \det h, x)<1$, the cohomology group $H^q(\mathbb{P}^n, K_{\mathbb{P}^n} \otimes Q\tensor Q)=H^q(\mathbb{P}^n, K_{\mathbb{P}^n} \otimes {\rm T}\mathbb{P}^n)$ vanishes for any positive integer $q>0$ from Corollary \ref{maincor2}. However, the Serre duality theorem implies that
$$
H^q(\mathbb{P}^n, K_{\mathbb{P}^n} \otimes {\rm T}\mathbb{P}^n)^\star \cong H^{n-q}(\mathbb{P}^n, {\rm T}^\star \mathbb{P}^n)\cong H^{(1, n-q)}(\mathbb{P}^n, \mathbb{C}) \cong
\begin{cases}
\mathbb{C} & (q=n-1) \\
0 & (q \neq n-1).
\end{cases}
$$
This contradicts to Corollary \ref{maincor2}. 
\end{proof}


\end{document}